\documentclass{article}

\usepackage{amsmath,amsthm,amssymb}
\usepackage{tikz}

\newtheorem{theorem}{Theorem}[section]
\newtheorem{lemma}[theorem]{Lemma}
\newtheorem{proposition}[theorem]{Proposition}
\newtheorem{corollary}[theorem]{Corollary}
\newtheorem{exam}[theorem]{Example}

\newtheorem{ques}{Question}

\newcommand{\ctm}{\mathfrak{c}}

\newcommand{\N}{\mathbb{N}}
\newcommand{\R}{\mathbb{R}}

\newcommand{\K}{\mathcal{K}}
\newcommand{\B}{\mathcal{B}}
\newcommand{\M}{\mathcal{M}}

\newcommand{\tq}{\ge_T}
\newcommand{\te}{=_T}

\newcommand{\omoneom}{(\omega_1,\omega)}
\newcommand{\CD}{\mathop{CD}}

\newcommand{\vx}{\mathbf{x}}
\newcommand{\vy}{\mathbf{y}}
\newcommand{\vz}{\mathbf{z}}
\newcommand{\vn}{\mathbf{n}}
\newcommand{\ssum}{\raisebox{0.3ex}{$\scriptstyle \sum$}}

\DeclareMathOperator{\supp}{supp}
\newcommand{\cl}[1]{\overline{#1}}
\newcommand{\down}[1]{\, \downarrow \! \! #1\,}
\newcommand{\type}[2]{$\langle${}#1,\,{#2}$\rangle$}

\title{Products of Directed Sets with Calibre $(\omega_1,\omega)$}

\author{Paul Gartside and Jeremiah Morgan}

\date{\today}

\begin{document}
\maketitle

\begin{abstract}
A directed set $P$ is calibre $\omoneom$ if every uncountable subset of $P$ contains an infinite bounded subset.  $P$ is productively calibre $\omoneom$ if $P \times Q$ is calibre $\omoneom$ for every directed set $Q$ with calibre $\omoneom$, and $P$ is powerfully calibre $\omoneom$ if the countable power of $P$ is calibre $\omoneom$.  

It is shown that (1) uncountable products are calibre $\omoneom$ only in highly restrictive circumstances, (2) many but not all $\ssum$-products of calibre $\omoneom$ directed sets are calibre $\omoneom$, (3) there are directed sets which are calibre $\omoneom$ but neither productively nor powerfully calibre $\omoneom$, and (4) there are directed sets which are powerfully but not productively calibre $\omoneom$. 

As an application, the position is established of $\sum \omega^{\omega_1}$ in the Tukey order among Isbell's classical 10 directed sets.


\end{abstract}

\section{Introduction} 

Among all the chain conditions on a directed set, only one, calibre $\omoneom$, is known to be non-productive \emph{in ZFC}. 
Indeed, Todorcevic showed in \cite{Tod} that if $S$ is a stationary and co-stationary subset of $\omega_1$, then $\K(S)$ and $\K(\omega_1\setminus S)$ each have calibre $\omoneom$, but $\K(S)\times \K(\omega_1\setminus S)$ does not have calibre $\omoneom$. 
Here, a directed set is said to be, or have, \emph{calibre $\omoneom$} if each of its uncountable subsets contains a bounded infinite subset, and $\K(X)$ denotes the directed set of all  compact subsets of a space $X$, ordered by inclusion. 
We investigate further the behavior of calibre $\omoneom$ in products. 

Specifically, we show uncountable products of directed sets are only rarely calibre $\omoneom$ (Theorem~\ref{unctble_prod}), and a wide variety of, but not all, $\ssum$-products are calibre $\omoneom$ (Theorem~\ref{Sigma_cosmic_CSB} and Example~\ref{Sigma_ctble}). 
We then turn to finite and countable products. 
Call a directed set $P$  \emph{productively calibre $\omoneom$} if $P \times Q$ is calibre $\omoneom$ for every directed set $Q$ with calibre $\omoneom$, and call $P$ \emph{powerfully calibre $\omoneom$} if the countable power of $P$ is calibre $\omoneom$. 
We show there are directed sets which are calibre $\omoneom$ but neither productively nor powerfully calibre $\omoneom$ (Example~\ref{ex:neither}) and directed sets which are powerfully but not productively calibre $\omoneom$ (Example~\ref{Sigma_ctble}). An example of a productively   but not powerfully calibre $\omoneom$ directed set remains elusive.

\begin{ques}
Is there a directed set $P$ which is productively calibre $\omoneom$ but not powerfully calibre $\omoneom$?
\end{ques}

Our principle tools are the Tukey order and topology. 
If $P$ and $Q$ are directed sets, then $P$ \emph{Tukey quotients} to $Q$, denoted $P \tq Q$, if there is a map  $\phi:P\to Q$ carrying cofinal sets to cofinal sets. 
$P$ and $Q$ are \emph{Tukey equivalent}, denoted $P \te Q$, if both $P \tq Q$ and $Q \tq P$. If $P$ is calibre $\omoneom$ and $P \tq Q$ then $Q$ has calibre $\omoneom$ (see Lemma~\ref{Tukey_pres_cal}).  
The topology is used both to construct directed sets of the form $\K(X)$ with desired chain conditions (in fact, see \cite{FG}, every directed set is Tukey equivalent to some $\K(X)$) and to prove that specific directed sets are calibre $\omoneom$. A convenient result in the latter direction, following from Lemma~\ref{ce_to_om1om}, is that every directed set with a second countable topology in which every convergent sequence has a bounded subsequence  (abbreviated CSBS) is calibre $\omoneom$. 

In \cite{Isbell} Isbell introduced ten directed sets and discussed the relationships between them, establishing that among them at least seven were distinct up to Tukey equivalence. The ten directed sets were a first group of $\mathbf{1}$, $\omega$, $\omega_1$, $\omega \times \omega_1$, $[\omega_1]^{<\omega}$, and a second group, $\omega^\omega$, $\sum \omega^{\omega_1}$, $\text{NWD}$, $\mathcal{Z}_0$ and $\ell_1$.  The members of the first group are all of cofinality no more than $\omega_1$ and are easily distinguished. The first four members can be directly checked to be calibre $\omoneom$, in contrast to $[\omega_1]^{<\omega}$. 

Fremlin \cite{Frem_An} investigated the relationships between the members of Isbell's second group \emph{excluding $\sum \omega^{\omega_1}$}. He also introduced $\mathcal{E}_\mu$, the ideal of all compact subsets of $I$ which are null sets (with respect to the standard Lebesgue measure, $\mu$). 
The task of determining all relations between these five directed sets ($\omega^\omega$, $\mathcal{E}_\mu$, $\text{NWD}$, $\mathcal{Z}_0$ and $\ell_1$)  was only recently completed, see \cite{TodSol}, \cite{Matrai}. Critical to all of these results is the fact that these five directed sets have natural topologies which are second countable and CSBS, and hence they are calibre $\omoneom$.

To date, the position of $\sum \omega^{\omega_1}$ -- the last of Isbell's second five -- was unknown. We complete this line of investigation by showing that $\sum \omega^{\omega_1}$ is strictly above $\omega^\omega$, and incomparable with each of $\mathcal{E}_\mu$, $\text{NWD}$, $\mathcal{Z}_0$ and $\ell_1$. 
Clearly $\sum\omega^{\omega_1} \tq \omega^\omega$, so our claims follow from Proposition~\ref{no_csbs_above}, which says that no second countable CSBS directed set Tukey quotients to  $\sum\omega^{\omega_1}$, and Proposition~\ref{Emu}.

The paper is organized as follows. Section~\ref{sec:pre} covers necessary background material on general Tukey quotients and calibres, topology and topological directed sets, and `types' of directed sets (which allow us to break a directed set into simpler pieces). 
In Section~\ref{sec:prod_thms} we establish all the positive results on calibre $\omoneom$ in products, starting with `large' products and moving through $\ssum$-products, to finite and countable products. 
Section~\ref{sec:exs}, on the other hand, presents our examples. 
In addition to the `not powerful nor productive', `powerful but not productive' and $\ssum$-product examples mentioned above, we start by applying our positive results to two examples from \cite{LV}.
The results on Isbell's directed sets are given in Section~\ref{sec:isbell}.

\section{Preliminaries}\label{sec:pre}

\subsection{Tukey Quotients and Calibres}
A subset of a partially ordered set is \emph{bounded} if it has an upper bound. A partially ordered set $P$ is called a \emph{directed set} if each finite subset of $P$ is bounded.  We call $(P',P)$ a \emph{directed set pair} if $P$ is a directed set and $P'$ is a subset of $P$.  A subset $C$ of $P$  is \emph{cofinal for $(P',P)$} if for each $p'$ in $P'$, there is a $c$ in $C$ with $p'\le c$.

A directed set pair $(Q',Q)$ is called a \emph{Tukey quotient} of the pair $(P',P)$, written $(P',P) \tq (Q',Q)$, if there is a map $\phi : P \to Q$ which takes cofinal sets for $(P',P)$ to cofinal sets for $(Q',Q)$.  If $P'=P$, then we abbreviate the directed set pair $(P',P)$ to just $P$.  Thus, for example, the relation $(P,P)\tq (Q,Q)$ will be written as $P\tq Q$, and this coincides with the usual notion of a Tukey quotient of directed sets.  If $(P',P)\tq (Q',Q)$ and $(Q',Q)\tq (P',P)$, then we say the pairs are \emph{Tukey equivalent} and write $(P',P) =_T (Q',Q)$. 
 A directed set is called \emph{Dedekind complete} if each bounded subset has a least upper bound.  A map $\phi : P\to Q$ between directed sets is called \emph{order-preserving} if $\phi(p_1)\le \phi(p_2)$ in $Q$ whenever $p_1\le p_2$ in $P$.  See \cite{PG_AM} for a proof of the following lemma.

\begin{lemma} \label{tq_dc}
Let $(P',P)$ and $(Q',Q)$ be directed set pairs.

(i) If there is an order-preserving map $\phi: P\to Q$ such that $\phi(P')$ is cofinal for $(Q',Q)$, then $(P',P)\tq (Q',Q)$.

(ii) The converse of (i) holds if $Q$ is Dedekind complete.
\end{lemma}

For any cardinals $\kappa\ge\lambda\ge\mu\ge\nu$, we say a directed set pair $(P',P)$ has \emph{calibre $(\kappa,\lambda,\mu,\nu)$} (or $P'$ has \emph{relative calibre $(\kappa,\lambda,\mu,\nu)$ in $P$}) if each $\kappa$-sized subset of $P'$ contains a $\lambda$-sized subset $S$ such that each $\mu$-sized subset of $S$ contains a $\nu$-sized subset which is bounded in $P$.  If $\mu=\nu$, then we write simply `calibre $(\kappa,\lambda,\mu)$' Similarly, if $\lambda=\mu=\nu$, then we write `calibre $(\kappa,\lambda)$', and we write `calibre $\kappa$' if $\kappa=\lambda=\mu=\nu$.  As before, we abbreviate the directed set pair $(P,P)$ to just $P$, and thus we may speak of the calibres of a directed set $P$. 
It is known (see \cite{KCal}, for example) that $\omega^\omega$, with its usual product order, has calibre $\omega_1$ if and only if $\omega_1 < \mathfrak{b}$, where $\mathfrak{b}$ is the minimal size of an unbounded set in $\omega^\omega$ in the mod-finite order.

The relationship between Tukey quotients and calibres is expressed in the next lemma, which says that Tukey quotients preserve calibres.

\begin{lemma} \label{Tukey_pres_cal}
Suppose $\kappa$ is a regular cardinal and $(P',P)$ has calibre $(\kappa,\lambda,\mu,\nu)$.  If $(P',P)\tq (Q',Q)$, then $(Q',Q)$ also has calibre $(\kappa,\lambda,\mu,\nu)$.
\end{lemma}

The proof of the preceding lemma uses an equivalent condition for $(P',P)\tq (Q',Q)$, namely the existence of a so-called \emph{Tukey map} $\psi : Q'\to P'$ such that $\psi(E)$ is unbounded in $P$ for any subset $E$ of $Q'$ that is unbounded in $Q$ (see \cite{PG_AM}).  
In fact, the proof of the preceding lemma only requires that $\psi$ maps unbounded sets of size $\mu$ to unbounded sets.  We omit the proof.

\begin{lemma} \label{muTukey_pres_cal}
Let $(P',P)$ and $(Q',Q)$ be directed set pairs such that $(P',P)$ has calibre $(\kappa,\lambda,\mu,\nu)$, where $\kappa$ is a regular cardinal.  If there is a map $\psi : Q' \to P'$ such that for any $\mu$-sized subset $E$ of $Q'$ which is unbounded in $Q$, the image $\psi(E)$ is unbounded in $P$, then $(Q',Q)$ also has calibre $(\kappa,\lambda,\mu,\nu)$. 
\end{lemma}



\subsection{Topology and Topological Directed Sets}

A space is \emph{second countable}, abbreviated $2^o$, if it has a countable base; it is \emph{first countable} if each point has a countable local base; and it has \emph{countable extent} if every closed discrete subset is countable. Other topological terminology and notation are standard, see \cite{Eng}, and are introduced in place. 

A \emph{topological directed set} is a directed set with a topology.  For example, if $X$ is any space, then the set $\K(X)$ of all compact subsets of $X$ is directed by inclusion and is naturally equipped with the Vietoris topology.  A topological directed set $P$ is said to be \emph{CSB}  if every convergent sequence in $P$ is bounded, \emph{CSBS}  if every convergent sequence in $P$ contains a bounded subsequence, \emph{KSB}  if every compact subset of $P$ is bounded, and \emph{DK}   if every down set, $\down{p} = \{p' \in P : p' \le p\}$, of $P$ is compact. Clearly KSB implies CSB and CSB implies CSBS.

Fundamental properties of the Vietoris topology on $\K(X)$, for Hausdorff $X$, include: (i) if $X$ is compact then $\K(X)$ is compact and (ii) if $\mathcal{K}$ is a compact subset of $\K(X)$ then $\bigcup \K$ is a compact subset of $X$. The next lemma follows.
\begin{lemma} \label{KX_basics}
Let $X$ be a Hausdorff space. Then $\K(X)$ is KSB, DK, and Dedekind complete.
\end{lemma}

\begin{lemma} \label{prod_KM}
For any spaces $X_\alpha$, we have $\prod_\alpha \K(X_\alpha) =_T \K(\prod_\alpha X_\alpha)$.
\end{lemma}
\begin{proof}
The relevant Tukey quotients are $(K_\alpha)_\alpha \mapsto \prod_\alpha K_\alpha$ and $K \mapsto (\pi_\alpha[K])_\alpha$.
\end{proof}

\begin{lemma} \label{om1om_to_ce}
Let $Q$ be a topological directed set. If $Q$ is Hausdorff, DK and has calibre $(\omega_1,\omega)$, then $Q$ has countable extent.
\end{lemma}

\begin{proof}
Suppose $S$ is an uncountable closed discrete subset of $Q$.  Then $S$ contains an infinite subset $S_1$ with an upper bound $q$ in $Q$.  Then $\down{q}\cap S$ is compact, infinite, and closed discrete, which is a contradiction.
\end{proof}

\begin{lemma} \label{ce_to_om1om}
Let $Q$ be a topological directed set. Then $Q$ has calibre $(\omega_1,\omega)$ if:

(1) $Q$ is first countable, has countable extent and is CSBS, or

(2) $Q$ is locally compact, has countable extent and is KSB.
\end{lemma}

\begin{proof}
Let $S$ be an uncountable subset of $Q$.  Since $S$ is not closed and discrete, then there is a $q$ in $Q$ such that $q$ is in the closure of $S\setminus\{q\}$.  For (1), since $Q$ is first countable, then there is a sequence $(q_n)_n$ in $S\setminus\{q\}$ that converges to $q$.  Now, CSBS implies that $Q$ contains an upper bound $u$ for some subsequence $(q_{n_k})_k$, so $S$ contains an infinite subset $S_0=\{q_{n_k} : k\in\N\}$ with an upper bound in $Q$. For (2), since $Q$ is locally compact there is a compact neighborhood $C$ of $q$. Then $C$ contains an infinite subset $S_0$ of $S$, and $\cl{S_0}$ is therefore compact, so by KSB, $S_0$ has an upper bound.  
\end{proof}


\begin{corollary} \label{kx_ce}
Let $X$ be locally compact and Hausdorff. Then $\K(X)$ has calibre $(\omega_1,\omega)$ if and only if $\K(X)$ has countable extent.
\end{corollary}

Basic open sets in $\K(X)$  have the form $\langle U_1,\ldots, U_n \rangle = \{K\in\K(X) : K \subseteq \bigcup_{i=1}^n U_i,\ K\cap U_i \neq \emptyset \ \forall i=1,\ldots,n\}$ for some $n$ in $\N$ and $U_i$ open in $X$. The open sets $U_1, \ldots, U_n$ can be assumed to come from any specified base for $X$.

\begin{lemma} \label{kx_1ctble}
If $X$ is first countable and totally imperfect, then $\K(X)$ is also first countable.
\end{lemma}

\begin{proof}
Let $K \in \K(X)$.  Then we can enumerate $K = \{x_i : i\in\N\}$ since $X$ is totally imperfect.  For each $i\in\N$, fix a countable neighborhood base $(B_{i,n})_n$ for $x_i$.  For any finite sequence of positive integers $\vn = (n_1,\ldots,n_k)$, define $T_\vn = \langle B_{1,n_1},\ldots, B_{k,n_k} \rangle$.  Now let $U = \langle U_1,\ldots,U_\ell \rangle$ be any basic neighborhood of $K$, so $K$ is contained in $\bigcup_{j=1}^\ell U_j$ and intersects each $U_j$.

For each $i\in\N$, pick $n_i\in\N$ such that $x_i \in B_{i,n_i} \subseteq \bigcap \{U_j : x_i \in U_j\}$.  Since $K$ is compact, there is a $k\in\N$ such that $\{B_{i,n_i} : i=1,\ldots,k\}$ covers $K$.  Since $K$ intersects each $U_j$, then by choosing $k$ large enough, we can also ensure that for each $j\in\{1,\ldots,\ell\}$, there is some $i\in\{1,\ldots,k\}$ such that $B_{i,n_i} \subseteq U_j$.  For $\vn = (n_1,\ldots,n_k)$, it follows that $K \in T_\vn \subseteq U$.  Hence, $K$ has a countable neighborhood base: $\{T_\vn : \vn \in \bigcup_m \N^m,\ K\in T_\vn\}$.
\end{proof}

\subsection{Types of Directed Sets}

Let $Q$ be a directed set and $R$ a property of subsets of a directed set.  We say a family $\{S_q : q\in Q\}$ of subsets of a set $S$ is \emph{$Q$-ordered} if $S_q \subseteq S_{q'}$ whenever $q\le q'$ in $Q$.  We say a directed set pair $(P',P)$ is of \emph{type} \type{$Q$}{$R$} if $P'$ is covered by a $Q$-ordered family $\{P_q : q \in Q\}$ of subsets of $P$ such that each $P_q$ has property $R$.  A directed set $P$ is of type \type{$Q$}{$R$} if $(P,P)$ is of type \type{$Q$}{$R$}.

If $\mathcal{Q}$ is a class of directed sets, then \type{$\mathcal{Q}$}{$R$} denotes the class of all directed sets with type \type{$Q$}{$R$} for some $Q$ in $\mathcal{Q}$.  Likewise, if $R'$ is a property of directed sets, then \type{$R'$}{$R$} denotes the class of all directed sets with type \type{$Q$}{$R$} for some directed set $Q$ with property $R'$.

Let $R$ be a property of subsets of a directed set, and let $P$ be a directed set.  Then $R(P)$ will denote the family of all subsets of $P$ with property $R$, ordered by inclusion.  If $R$ is preserved by finite unions, then $R(P)$ is a directed set.  Let $P'$ be a subset of $P$.  If each singleton subset of $P'$ has property $R$, then we can identify $P'$ with a subset of $R(P)$ (ignoring the partial order $P'$ inherits from $P$).  In this way, $(P',R(P))$ is a directed set pair.

In many cases throughout this paper, the family $R(P)$ is Dedekind complete.  For example, if $R$ is the property of having relative calibre $(\kappa,\lambda,\mu)$ in $P$, and if $\mathcal{A}$ is a bounded subset of $R(P)$, then $\bigcup \mathcal{A}$ is the least upper bound of $\mathcal{A}$ in $R(P)$.

\begin{lemma} \label{tq_types}\label{cd_type_tq}
Let $(P',P)$ be a directed set pair, let $Q$ and $Q'$ be directed sets, and let $R$ be a property of subsets of a directed set such that $(P',R(P))$ is a directed set pair, as described above.  Suppose also that $R(P)$ is Dedekind complete.  Then we have:

(i) $(P',P)$ has type \type{$Q$}{$R$} if and only if $Q \tq (P', R(P))$.

(ii) If $Q\tq Q'$, then type \type{$Q'$}{$R$} implies type \type{$Q$}{$R$}.

(iii) If $Q'$ is Dedekind complete, then (ii) holds even if $R(P)$ is not Dedekind complete.
\end{lemma}

\begin{proof}
For (i), notice that $(P',P)$ has type \type{$Q$}{$R$} precisely when there is an order-preserving map $\phi : Q \to R(P)$ whose image covers $P'$, and since the points in $P'$ are identified with singletons in $R(P)$, then covering $P'$ is the same as being cofinal for $P'$ in $R(P)$.  Statement (i) now follows from Lemma~\ref{tq_dc}, and then statement (ii) follows from (i) and transitivity of the Tukey order.

To prove (iii), let $\phi_1 : Q \to Q'$ be an order-preserving map witnessing $Q\tq Q'$, which exists by Lemma~\ref{tq_dc}.  If $(P',P)$ has type \type{$Q'$}{$R$}, then there is an order-preserving map $\phi_2 : Q' \to R(P)$ whose image covers $P'$.  The map $\phi_1 \circ \phi_2$ then witnesses that $(P',P)$ also has type \type{$Q$}{$R$}.
\end{proof}




\begin{lemma} \label{trivial_types}
Let $R$ be a property shared by all directed sets that have a maximal element.  If there is an order-preserving map witnessing $Q\tq (P',P)$, then $(P',P)$ has type \type{$Q$}{$R$}. 
In particular, $P$ has type \type{$P$}{calibre $(\kappa,\lambda,\mu,\nu)$} for any cardinals $\kappa\ge\lambda\ge\mu\ge\nu$.
\end{lemma}

\begin{proof}
Fix an order-preserving map $\phi : Q \to P$ whose image is cofinal for $(P',P)$.  The family $\{\down{\phi(q)} : q\in Q\}$ is $Q$-ordered and covers $P'$.  Also, each $\down{\phi(q)}$ has a maximal element and therefore has property $R$.
\end{proof}

\begin{lemma} \label{calibre_types}
Let $\kappa$ be a regular cardinal, and suppose $\kappa \ge \lambda \ge \mu \ge \nu$.  If $Q$ has calibre $(\kappa,\lambda)$ and $P$ has type \type{$Q$}{relative calibre $(\mu,\nu)$}, then $P$ has calibre $(\kappa,\lambda,\mu,\nu)$.
\end{lemma}

\begin{proof}
Let $\{P_q : q\in Q\}$ be a $Q$-ordered cover of $P$, where $Q$ has calibre $(\kappa,\lambda)$ and each $P_q$ has relative calibre $(\mu,\nu)$ in $P$.  Let $S$ be a subset of $P$ with size $\kappa$.  For each $p\in S$, pick $q_p\in Q$ such that $p \in P_{q_p}$.  If $\{q_p : p\in S\}$ has size $\kappa$, then we can find a $\lambda$-sized subset $S'$ of $S$ such that $\{q_p : p\in S'\}$ has an upper bound $q' \in Q$.  If instead $\{q_p : p\in S\}$ has size less than $\kappa$, then since $\kappa$ is regular, we can find a $\kappa$-sized subset $S''$ of $S$ and a $q''$ in $Q$ such that $q_p = q''$ for all $p\in S''$.

In either case, we obtain a $q\in Q$ such that $S_q = S\cap P_q$ has size at least $\lambda$.  Complete the proof  using the fact that $P_q$ has relative calibre $(\mu,\nu)$ in $P$.
\end{proof}

\section{Product Theorems}\label{sec:prod_thms}

\subsection{Uncountable Products}
Uncountable products have calibre $(\omega_1,\omega)$ only in trivial circumstances.

\begin{theorem}\label{unctble_prod}
Let $\kappa$ be an  uncountable cardinal and $\{P_\alpha : \alpha < \kappa\}$ be a family of directed sets.

Then $\prod \{P_\alpha : \alpha < \kappa\}$ has calibre $(\omega_1,\omega)$ if and only if all countable subproducts have calibre $(\omega_1,\omega)$ and all but countably many $P_\alpha$ are countably directed.
\end{theorem}
\begin{proof}
If $\prod \{P_\alpha : \alpha < \kappa\}$ has calibre $(\omega_1,\omega)$ then certainly (as projections are Tukey quotients) all countable subproducts have calibre $(\omega_1,\omega)$. And note that a $P_\alpha$ is not countably directed if and only if $P_\alpha \ge_T \omega$. So if uncountably many $P_\alpha$'s are not countably directed then $\prod \{P_\alpha : \alpha < \kappa\} \ge_T \omega^{\omega_1}$, and apply the next lemma.

For the converse let $C=\{ \alpha : P_\alpha$ is not countably directed$\}$. By hypothesis $C$ is countable and $\prod \{P_\alpha : \alpha \in C\}$ has calibre $(\omega_1,\omega)$. Since $P_\alpha$ is countably directed for each $\alpha \notin C$, it is clear that $\prod \{P_\alpha : \alpha \notin C\}$ is countably directed, or equivalently has calibre $\omega$. So $\prod \{P_\alpha : \alpha \in \kappa\} = \prod \{P_\alpha : \alpha \in C\} \times \prod \{P_\alpha : \alpha \notin C\}$, as the product of a calibre $(\omega_1,\omega)$ and a calibre $\omega$ directed set, has calibre $(\omega_1,\omega)$.
\end{proof}

\begin{lemma}
The directed set $\omega^{\omega_1}$ does not have calibre $(\omega_1,\omega)$.
\end{lemma}
\begin{proof}
For each $\alpha < \omega_1$ fix an injection $b_\alpha : \alpha \to \omega$, and then define $f_\alpha : \omega_1 \to \omega$ by $f_\alpha (\beta)=b_\beta (\alpha)+1$ if $\beta > \alpha$ and otherwise zero. 
Then $\{f_\alpha : \alpha <\omega_1\}$ is uncountable (indeed, the $f_\alpha$ are distinct since the smallest $\beta$ with $f_\alpha(\beta)\neq 0$ is $\beta = \alpha+1$) but every infinite subset  is unbounded. To see this take any countably infinite subset $S$ of $\omega_1$. Let $\alpha_\infty = \sup S +1$. Then for $\alpha$ in $S$ we have $\alpha <\alpha_\infty$, so $f_\alpha (\alpha_\infty)=b_{\alpha_\infty}(\alpha)+1$. Since $S$ is infinite and $b_{\alpha_\infty}$ is injective, we see the $f_\alpha$'s for $\alpha$ in $S$ are unbounded on $\alpha_\infty$. 
\end{proof}

\subsection{$\ssum$-Products}
$\ssum$-products are  directed subsets of full products which are large, but not `too large' to immediately preclude them from being calibre $\omoneom$.  In particular, they do not project onto an uncountable subproduct of the full product. 

For the remainder of this section, every directed set will be assumed to have a minimum element. If $P$ is a directed without a minimum, then we can add one, and this changes no Tukey properties of $P$. In the case when $P$ is a topological directed set, we can make the new minimum element isolated, and this changes no (relevant) topological properties of $P$.
Given a family, $\{P_\alpha : \alpha < \kappa\}$ of directed sets, where each $P_\alpha$ has minimum element $0_\alpha$, the $\ssum$-product of the family, denoted $\sum_\alpha P_\alpha$, is the subset $\{(p_\alpha)_\alpha   : \supp {(p_\alpha)_\alpha} \text{ is countable}\}$  of $\prod_\alpha P_\alpha$ with the product order. Here, $\supp{(p_\alpha)_\alpha} = \{\alpha \in \kappa : p_\alpha \ne 0_\alpha\}$ is the support. We write $\sum P^\kappa$ for $\sum_\alpha P_\alpha$ if each directed set $P_\alpha$ in the family is (Tukey equivalent to) $P$. 

A space $X$ is \emph{cosmic} if it has a countable network of closed sets; in other words, there is a countable collection, $\mathcal{C}$, of closed subsets such that whenever some point $x$ of $X$ is in an open set $U$, there is a $C$ from $\mathcal{C}$ such that $x \in C \subseteq U$. Among $T_3$ spaces, a space is cosmic if and only if it is the continuous image of a separable metrizable space. 
Evidently every second countable and $T_3$ space is cosmic. 

\begin{theorem}\label{Sigma_cosmic_CSB}
Let $\{P_\alpha : \alpha < \kappa\}$ be a family of cosmic topological directed sets with CSBS.
Then  $\sum_\alpha P_\alpha$ has calibre $(\omega_1,\omega)$. 
\end{theorem}

\begin{proof}
 First, we simplify the factors topologically. Suppose $P$ is a topological directed set which is cosmic and has CSBS. 
 Take a countable network for $P$ of closed sets.
 Refine the topology on $P$ to get a topology $\tau$, by taking the elements of the network as clopen sets. 
 Then $(P,\tau)$ is zero-dimensional, separable metrizable, and also has CSBS.
Without loss of generality, then, we assume all the $P_\alpha$ are subspaces of the Cantor set.

Let $S = \{0,1\}^{<\omega}$, which is the set of all finite sequences of zeros and ones. For each $s$ in $S$, let $|s|$ denote the length of $s$, which is given by $|s| = 0$, if $s$ is the empty sequence, and $|(s_0, \ldots,s_{n-1})|=n$.  For each $\sigma = (s_n)_n$ in the Cantor set, $\{0,1\}^\omega$, and for each $n<\omega$, let $\sigma|n = (s_0,\ldots,s_{n-1})$.  Then the Cantor set has a base of clopen sets $\mathcal{B}=\bigcup_n \mathcal{B}_n$, where $\mathcal{B}_n = \{ B_s : s\in S \text{ and } |s|=n\}$ and $B_s = \{\sigma\in \{0,1\}^\omega : \sigma|n = s\}$.  Now, for any finite subset $F$ of $\kappa$, and for any $n<\omega$, define $\mathcal{C}(F,n)= \left\{ \left( \sum P_\alpha \right) \cap \left( \bigcap_{\alpha \in F}  \pi_{\alpha}^{-1} U_\alpha \right) : U_\alpha \in \mathcal{B}_n \text{ for each } \alpha\in F \right\}$.  Note that since each $\mathcal{B}_n$ is finite, then $\mathcal{C}(F,n)$ is a finite clopen partition of $\sum P_\alpha$.

Take any uncountable subset $E$ of $\sum P_\alpha$. We show some infinite subset of $E$ has an upper bound.  For each subset $G$ of $\sum P_\alpha$, if $G$ meets $E$ then we select a point $e(G)$ in $G \cap E$.  Also, for each $e = (e_\alpha)_\alpha$ in $E$, fix a surjection $f_e : \N \to \supp (e) = \{\alpha<\kappa : e_\alpha\neq 0_\alpha\}$.

Let $E_1$ be the singleton $\{e(\sum P_\alpha)\}$.  Now, inductively define $A_n = \{f_e(i) : e\in E_n,\ 1\le i\le n\}$, $\mathcal{C}_n = \mathcal{C}(A_n,n)$, and $E_{n+1} = E_n\cup \{e(C\setminus E_n) : C\in \mathcal{C}_n \text{ and } (C\setminus E_n)\cap E \neq \emptyset\}$.  Then the $E_n$'s are finite subsets of $E$, the $A_n$'s are finite subsets of $\kappa$, and each $\mathcal{C}_n$ is a finite clopen partition of $\sum P_\alpha$.  We set $E' = \bigcup_n E_n$, $A = \bigcup_n A_n$, and $\mathcal{C} = \bigcup \mathcal{C}_n$.  Note that $A$ is the union of $\{\supp(e) : e \in E'\}$.  Observe that the projections, $\{ \pi_A C : C \in \mathcal{C}\}$, of $\mathcal{C}$ into $P=\prod \{ P_\alpha : \alpha \in A\}$ form a base for the latter space.

Since $E'$ is countable and $E$ is uncountable, we can pick $e$ in $E\setminus E'$.  For each $n$, there is a (unique) $C_n$ in $\mathcal{C}_n$ such that $e$ is in $C_n$. Thus, $(C_n \setminus E_n)\cap E$ is nonempty for each $n$, so by the inductive construction above, the point $e_n = e(C_n \setminus E_n)$ is in $E_{n+1} \setminus E_n$.  Therefore, $E'' = \{e_n : n\in\N\}$ is an infinite subset of $E$.

Write $e = (e_\alpha)_\alpha$, and let $x = (x_\alpha)_\alpha$ be the point in $\sum P_\alpha$ given by $x_\alpha = e_\alpha$ for each $\alpha$ in $A$ and $x_\alpha = 0_\alpha$ otherwise.  The sets $\pi_A C_n$ form a decreasing local base at $\pi_A(x) = \pi_A(e)$ in $P=\prod \{P_\alpha : \alpha \in A\}$ , so the $\pi_A(e_n)$'s converge in $P$ to $\pi_A(x)$. Enumerate $A=\{\alpha_n : n \in \N\}$. Let $S_0=\N$. Recursively, using the fact that $P_{\alpha_i}$ is CSBS for each $i$, find an infinite subset $S_{i}$ of $S_{i-1}$ such that $\{\pi_{\alpha_{i}} (e_n) : n \in S_{i}\}$ has an upper bound $p_{\alpha_{i}}'$ in $P_{\alpha_{i}}$. Pick distinct $n_i$ in $S_i$ for $i \in \N$. For each $i$, let $p_{\alpha_i}$ be an upper bound of $p_{\alpha_i}'$ and $\pi_{\alpha_i}(e_{n_j})$ for $j<i$. For $\alpha \notin A$ let $p_\alpha=0_\alpha$.  Then $p=(p_\alpha)_\alpha$ is an upper bound for the infinite set $E'''=\{e_{n_i} : i \in \N\}$ in $\sum P_\alpha$. 
\end{proof}




We can generalize Theorem~\ref{Sigma_cosmic_CSB} by applying the next result.

\begin{theorem}
Let $\mathcal{Q}$ be a class of directed sets.  The following are equivalent:

(i) for any family $\{P_\alpha : \alpha<\kappa\}$ of directed sets in the class $\mathcal{Q}$, the $\ssum$-product $\sum P_\alpha$ has calibre $\omoneom$, and

(ii) for any family $\{P_\alpha : \alpha<\kappa\}$ of directed sets in the class \type{$\mathcal{Q}$}{countably directed}, the $\ssum$-product $\sum P_\alpha$ has calibre $\omoneom$.
\end{theorem}

\begin{proof}
By Lemma~\ref{trivial_types}, a directed set $Q$ always has type \type{$Q$}{countably directed}, so $\mathcal{Q}$ is contained in \type{$\mathcal{Q}$}{countably directed}, and (ii) immediately implies (i).

Assume (i) and let $\{P_\alpha : \alpha<\kappa\}$ be a family of directed sets in the class \type{$\mathcal{Q}$}{countably directed} such that each $P_\alpha$ has a minimal element $0_\alpha$.  So for every $\alpha<\kappa$, we have a directed set $Q_\alpha$ in $\mathcal{Q}$ such that $P_\alpha = \bigcup \{P_{\alpha,q} : q\in Q_\alpha\}$, where each $P_{\alpha,q}$ is countably directed and $P_{\alpha,q_1} \subseteq P_{\alpha,q_2}$ whenever $q_1\le q_2$ in $Q_\alpha$.  We may assume that $Q_\alpha$ contains a minimal element $0'_\alpha$ and that $P_{\alpha,0'_\alpha} = \{0_\alpha\}$.

Fix a map $\psi_\alpha: P_\alpha \to Q_\alpha$ for each $\alpha<\kappa$ such that $p\in P_{\alpha,\psi_\alpha(p)}$ for each $p$ in $P_\alpha$.  We can choose $\psi_\alpha$ so that $\psi_\alpha(0_\alpha) = 0'_\alpha$.  Then let $\psi = \prod_\alpha \psi_\alpha : \prod_\alpha P_\alpha \to \prod_\alpha Q_\alpha$, which restricts to a map from $\sum P_\alpha$ into $\sum Q_\alpha$.  By (i), we know that $\sum Q_\alpha$ has calibre $(\omega_1,\omega)$, so according to Lemma~\ref{muTukey_pres_cal}, it suffices to show that for any unbounded countable subset $E$ of $\sum P_\alpha$, its image $\psi(E)$ is unbounded in $\sum Q_\alpha$.

Let $E$ be a countable subset of $\sum P_\alpha$, and suppose $\psi(E)$ is bounded (above) by some $q = (q_\alpha)_\alpha$ in $\sum Q_\alpha$.  Then for every $e = (e_\alpha)_\alpha$ in $E$ and any $\alpha<\kappa$, we have $\psi_\alpha(e_\alpha) \le q_\alpha$ and so $e_\alpha \in P_{\alpha,\psi_\alpha(e_\alpha)} \subseteq P_{\alpha,q_\alpha}$.  As $P_{\alpha,q_\alpha}$ is countably directed, there is a $p_\alpha$ in $P_{\alpha,q_\alpha}$ such that $e_\alpha \le p_\alpha$ for each $e\in E$.  Thus, $p=(p_\alpha)_\alpha$ is an upper bound for $E$ in $\prod_\alpha P_\alpha$.  Moreover, since $P_{\alpha,0'_\alpha} = \{0_\alpha\}$ for every $\alpha$, then $\supp(p) \subseteq \supp(q)$, and so $p$ is in $\sum P_\alpha$.
\end{proof}

\begin{corollary}
    Let $\{P_\alpha : \alpha < \kappa\}$ be a family of directed sets of type \type{cosmic and CSBS}{countably directed}.
Then  $\sum_\alpha P_\alpha$ has calibre $(\omega_1,\omega)$. 
\end{corollary}

\subsection{Finite and Countable Products}

In this section we present sufficient conditions for finite or countable products to be calibre $\omoneom$, and deduce sufficient conditions for directed sets to be productively or powerfully calibre $\omoneom$. We start by applying the topological approach. We then look at how to use sufficiently strong calibres to get productivity. Next, we use types to `boost' powerful or productive calibre $\omoneom$ from pieces of a directed set to the whole directed set. Finally, we combine the topology and the types. 

Our first lemma can easily be proved directly by noting that a countable product of CSB [respectively, CSBS] topological directed sets is again CSB [CSBS], and we omit the details. But note that the key conclusion also follows from Theorem~\ref{Sigma_cosmic_CSB}.  
\begin{lemma} \label{prod_CSB}
The class of second countable directed sets with CSB [CSBS] is closed under countable products. Hence every CSBS second countable directed set is powerfully calibre $\omoneom$.
\end{lemma}

\begin{corollary}
If $M_n$ is a separable metrizable space for each $n$, then $\prod_n \K(M_n)$ has calibre $\omoneom$. 
In particular, every $\K(M)$, where $M$ is separable metrizable, is powerfully calibre $\omoneom$.
\end{corollary}

\begin{lemma} \label{ctble_prod_rel_cal}
If $\kappa$ is an infinite regular cardinal and $Q_n$ has relative calibre $(\kappa,\kappa,\omega)$ in $P_n$ for each $n<\omega$, then $\prod_n Q_n$ has relative calibre $(\kappa,\omega)$ in $\prod_n P_n$.

In particular, if each $P_n$ has calibre $(\kappa,\kappa,\omega)$, then $\prod_n P_n$ has calibre $(\kappa,\omega)$.
\end{lemma}

\begin{proof}
Let $\{\vx^\alpha = (x^\alpha_n)_{n<\omega} : \alpha<\kappa\} \subseteq \prod_{n<\omega} Q_n$ where $\vx^\alpha \neq \vx^\beta$ whenever $\alpha\neq \beta$.  Fix $n<\omega$ and suppose $S\subseteq \kappa$ has size $|S|=\kappa$.  If $\{x^\alpha_n : \alpha\in S\}$ has size less than $\kappa$, then  since $\kappa$ is regular, we can find a $\kappa$-sized $S'\subseteq S$ such that $\{x^\alpha_n : \alpha\in S'\}$ is a singleton.  If instead $|\{x^\alpha_n : \alpha\in S\}| = \kappa$, then as $Q_n$ has relative calibre $(\kappa,\kappa,\omega)$ in $P_n$, we can find a $\kappa$-sized $S' \subseteq S$ such that each infinite subset of $\{x^\alpha_n : \alpha \in S'\}$ has an upper bound in $P_n$.

Thus, we can inductively construct a decreasing sequence $(S_n)_{n<\omega}$ of $\kappa$-sized subsets of $\kappa$ such that, for any $n<\omega$, each infinite subset of $\{x^\alpha_n : \alpha\in S_n\}$ has an upper bound in $P_n$ (perhaps vacuously).  Fix distinct $\alpha_n\in S_n$, for each $n<\omega$.  Since $\{\alpha_k : k\ge n\} \subseteq S_n$, then we can find an upper bound $x^\infty_n\in P_n$ for $\{x^{\alpha_k}_n : k\ge n\}$.  Since $P_n$ is directed, then we can also find an upper bound $z_n$ for $\{x^\infty_n, x^{\alpha_0}_n, \ldots, x^{\alpha_{n-1}}_n\}$ in $P_n$, and so $\vz=(z_n)_{n<\omega}$ is an upper bound for $\{\vx^{\alpha_k} : k<\omega\}$ in $\prod_{n<\omega} P_n$.
\end{proof}

\begin{corollary} \label{om1_and_om_powerful}
Calibre $\omega_1$ and calibre $\omega$ each imply powerfully calibre $(\omega_1,\omega)$.
\end{corollary}

\begin{proof}
If $P$ has calibre $\omega_1$, then it also has calibre $(\omega_1,\omega_1,\omega)$, and so $P^\omega$ has calibre $\omoneom$ by Lemma~\ref{ctble_prod_rel_cal}.  On the other hand, if $P$ has calibre $\omega$, then Lemma~\ref{ctble_prod_rel_cal} shows that $P^\omega$ has calibre $\omega$, which implies calibre $\omoneom$.
\end{proof}

Note that `calibre $\omega$' is equivalent to `countably directed'. But the relative versions are not equivalent. 
\begin{lemma} \label{ctble_prod_types}
(i) If $\kappa$ is an infinite regular cardinal and $P_n$ has type \type{$Q_n$}{[relative] calibre $(\kappa,\kappa,\omega)$} for each $n<\omega$, then $\prod_n P_n$ has type \type{$\prod_n Q_n$}{[relative] calibre $(\kappa,\omega)$}.

(ii) If $P_n$ has type \type{$Q_n$}{countably directed} for each $n<\omega$, then $\prod_n P_n$ has type \type{$\prod_n Q_n$}{countably directed}.
\end{lemma}

\begin{proof}
For (i), we prove the relative calibre case.  The proof of the calibre case can be obtained by omitting the parts in square brackets.  Say $P_n = \bigcup\{P_{n,q} : q\in Q_n\}$ where $P_{n,q}\subseteq P_{n,q'}$ whenever $q\le q'$ in $Q_n$ and each $P_{n,q}$ has [relative] calibre $(\kappa,\kappa,\omega)$ [in $P_n$].  For each $\mathbf{q} = (q_n)_{n<\omega}$ in $\prod_n Q_n$, define $R_\mathbf{q} = \prod_n P_{n,q_n}$.  Then $R_\mathbf{q}$ has [relative] calibre $(\kappa,\omega)$ [in $\prod_n P_n$] by Lemma~\ref{ctble_prod_rel_cal}, and $\{R_\mathbf{q} : \mathbf{q} \in \prod_n Q_n\}$ is a $(\prod_n Q_n)$-ordered cover of $\prod_n P_n$.

For (ii), a similar proof works, but note that if each $P_{n,q}$ is countably directed (in itself), then so is each $R_\mathbf{q}$.
\end{proof}

The weakest natural calibre that implies `productively calibre $\omoneom$' is calibre $(\omega_1, \omega_1, \omega, \omega)$. For completeness we include the proof. 
\begin{lemma}\label{calom1om1omom_prod}
If $P$ is calibre $(\omega_1, \omega_1, \omega, \omega)$, then $P$ is productively calibre $\omoneom$.
\end{lemma}

\begin{proof}
Let $Q$ be any directed set with calibre $\omoneom$, and let $S_1$ be any uncountable subset of $P\times Q$.  If $\pi_1[S_1]$ is a countable subset of $P$, then there must be a $p$ in $P$ such that $S_1 \cap \pi_1^{-1}\{p\}$ is uncountable, and then the fact that $Q$ has calibre $\omoneom$ immediately provides an infinite subset of $S_1 \cap \pi_1^{-1}\{p\}$ with an upper bound in $P\times Q$.  Thus, we may assume $\pi_1[S_1]$ is uncountable.

There is now an uncountable subset $P_1$ of $\pi_1[S_1]$ such that each infinite subset of $P_1$ contains an infinite subset with an upper bound in $P$.  Let $S_2 = S_1 \cap \pi_1^{-1}[P_1]$.  By the same idea as the previous paragraph, we may assume $\pi_2[S_2]$ is an uncountable subset of $Q$, and so $\pi_2[S_2]$ contains an infinite subset $Q_1$ with an upper bound in $Q$.

Let $S_3 = S_2 \cap \pi_2^{-1}[Q_1]$.  If $\pi_1[S_3]$ is finite, then it is bounded in $P$, so $S_3$ is an infinite subset of $S$ with an upper bound.  If instead $\pi_1[S_3]$ is infinite, then as a subset of $P_1$, $\pi_1[S_3]$ must contain an infinite set $P_2$ with an upper bound in $P$.  Then $S_4 = S_3 \cap \pi_1^{-1}[P_2]$ is an infinite subset of $S$ with an upper bound.
\end{proof}

\begin{lemma}
If $Q$ is productively calibre $\omoneom$ and $P$ has type \type{$Q$}{relative calibre $\omega$}, then $P$ is also productively calibre $\omoneom$.
\end{lemma}

\begin{proof}
Let $S$ be any directed set with calibre $(\omega_1,\omega)$.  By Lemma~\ref{trivial_types}, $S$ has type \type{$S$}{relative calibre $\omega$}, and so by Lemma~\ref{ctble_prod_types}, $P\times S$ has type \type{$Q\times S$}{relative calibre $\omega$}.  Since $Q\times S$ has calibre $(\omega_1,\omega)$, then by Lemma~\ref{calibre_types}, $P\times S$ does also.
\end{proof}

\begin{corollary}\label{sigma-prod}
If $P = \bigcup_{n<\omega} P_n$ where each $P_n$ has relative calibre $\omega$ in $P$, then $P$ is productively calibre $(\omega_1,\omega)$.
\end{corollary}

\begin{proof}
We may assume $P_n \subseteq P_{n+1}$ for each $n$, and so $P$ is of type \type{$Q$}{relative calibre $\omega$}, where $Q=\omega$ is productively calibre $(\omega_1,\omega)$.  Apply the previous lemma.
\end{proof}

\begin{corollary} \label{ctble_prod_type_class}
Let $\mathcal{Q}$ be a class of directed sets such that whenever $Q_n$ is in $\mathcal{Q}$ for each $n<\omega$, there is a $Q$ in $\mathcal{Q}$ with $Q \tq \prod_n Q_n$. Let $\mathcal{D}$ denote the class of all Dedekind complete directed sets.

(i) The class $\mathcal{D} \cap $\type{$\mathcal{Q}$}{countably directed} is closed under countable products.

(ii) If $\mathcal{Q}\subseteq \mathcal{D}$, then \type{$\mathcal{Q}$}{countably directed} is closed under countable products.
\end{corollary}

\begin{proof}
We prove (i) first.  Suppose that for each $n<\omega$, $P_n$ is Dedekind complete and has type \type{$Q_n$}{countably directed} where $Q_n$ is in $\mathcal{Q}$.  By Lemma~\ref{ctble_prod_types}, $P = \prod_n P_n$ has type \type{$Q'$}{countably directed} where $Q' = \prod_n Q_n$, and by assumption, there is a $Q$ in $\mathcal{Q}$ such that $Q\tq Q'$.  Since each $P_n$ is Dedekind complete, then so is $P$, and thus Lemma~\ref{cd_type_tq} implies that $P$ has type \type{$Q$}{countably directed}.

For (ii), we use the same notation as for (i).  If each $Q_n$ is Dedekind complete, then so is $Q'$.  Thus, even if $P$ is not Dedekind complete, we can use Lemma~\ref{cd_type_tq} to conclude that $P$ has type \type{$Q$}{countably directed}.
\end{proof}

\begin{corollary}\label{type_pow_rcalom}
Suppose $P_n$ has a type \type{$Q_n$}{relative calibre $\omega$} for each $n<\omega$. If $\prod_n Q_n$ has calibre $(\omega_1,\omega)$, then so does $\prod_n P_n$.  In particular:

(1) If $\kappa = \omega$ or $\kappa = \omega_1$, then any countable product of members of the class \type{calibre $\kappa$}{relative calibre $\omega$} has calibre $(\omega_1,\omega)$.

(2) Each member of the class \type{powerfully calibre $(\omega_1,\omega)$}{relative calibre $\omega$} is powerfully calibre $(\omega_1,\omega)$.
\end{corollary}

\begin{proof}
The main claim follows from Lemma~\ref{ctble_prod_types} (with $\kappa=\omega$) and Lemma~\ref{calibre_types} (with $\kappa=\omega_1$ and $\lambda=\mu=\nu=\omega$). Statement (1) then follows from Corollary~\ref{om1_and_om_powerful}.  Statement (2) is the special case of the main claim where the $P_n$'s are all the same, and the $Q_n$'s are also all the same.
\end{proof}

\begin{corollary}\label{sigma_pow}
If $P=\bigcup_{n<\omega} P_n$ where each $P_n$ has relative calibre $\omega$ in $P$, then $P$ is powerfully calibre $(\omega_1,\omega)$.
\end{corollary}

\begin{proof}
We may assume $P_n\subseteq P_{n+1}$ for each $n<\omega$.  Then $P$ has type \type{$Q$}{relative calibre $\omega$}, where $Q=\omega$ is powerfully calibre $(\omega_1,\omega)$, so we may apply (2) in the previous corollary.
\end{proof}

\begin{lemma}
Each member of the class \type{$2^o$+CSB}{countably directed} has calibre $\omoneom$.  In particular, each member of the class \type{$\K(\M)$}{countably directed} has calibre $\omoneom$.

Additionally, if $\mathcal{D}$ denotes the class of all Dedekind complete directed sets, then the following classes are closed under countable products:

(1) $\mathcal{D} \cap $\type{$2^o$+CSB}{countably directed}

(2) \type{$(2^o$+CSB$) \cap \mathcal{D}$}{countably directed}

(3) \type{$\K(\M)$}{countably directed}.
\end{lemma}

\begin{proof}
Lemma~\ref{ce_to_om1om} implies that every second countable and CSB directed set has calibre $(\omega_1,\omega)$.  Since countably directed implies (relative) calibre $\omega$, then Lemma~\ref{calibre_types} shows that the members of the class \type{$2^o$+CSB}{countably directed} have calibre $\omoneom$.

Closure under countable products follows from Corollary~\ref{ctble_prod_type_class}, Lemma~\ref{prod_CSB}, and Lemma~\ref{prod_KM}.
\end{proof}

\begin{lemma}\label{sigma_rel_csbs}
Let $P$ be a $2^o$ directed set such that $P=\bigcup_n P_n$ where, for each $n$, every sequence on $P_n$ with a limit in $P_n$ has [a subsequence with] an upper bound in $P$. Then we have: 

(a) $P$ is $2^o$ and CSB [CSBS] with respect to some finer topology.

(b) $P$ is type \type{$Q$}{relative calibre $\omega$} for every directed set $Q \tq (M,\K(M))$, where $M=\bigoplus_n P_n$. Thus $P$ is type \type{$\K(M)$}{relative calibre $\omega$}.

(c) Hence, from either (a) or (b), $P$ is powerfully calibre $\omoneom$.

In particular, if  each $P_n$ is $\sigma$-compact, then $P$ is type  \type{$\omega$}{relative calibre $\omega$}, while if  each $P_n$ is analytic, then $P$ is type \type{$\omega^\omega$}{relative calibre $\omega$}. Hence, if (i) each $P_n$ is $\sigma$-compact or (ii) each $P_n$ is analytic and $\omega_1<\mathfrak{b}$, then $P$ is  calibre $(\omega_1, \omega_1, \omega,\omega)$ and therefore productively calibre $\omoneom$. 
\end{lemma}

\begin{proof}
First we prove (a). Define $f : P \to \omega$ by $f(x)= \min \{n : x \in P_n\}$. Refine the given topology to a new $2^o$ topology $\sigma$ so that $f$ is continuous. We verify $(P,\sigma)$ is CSB [CSBS]. Take any sequence $(x_n)_n$  converging in $(P,\sigma)$ to $x$. Let $N=f(x)$. Then $x$ is in the open set $f^{-1} \{N\}$, which is contained in $P_N$. So all but finitely many terms of $(x_n)_n$ are in $P_N$. By hypothesis this sequence tail has [a subsequence with] an upper bound in $P$. Hence the full sequence [the subsequence] has an upper bound in $P$.

Now, to prove (b), fix an order-preserving  $\phi : Q \to \K(M)$ witnessing $Q \tq (M,\K(M))$. Define $\phi' : Q \to \K(P)$ by $\phi'(q) = \bigcup \{\phi(q) \cap P_n : n \in \N\}$, which is well-defined since for any compact subset $K$ of $M$, $K\cap P_n$ is nonempty for only finitely many $n$.  Then $\phi'$ is order-preserving and witnesses $Q \tq (P,\K(P))$. Further, each $\phi'(q)$ is relative calibre $\omega$ in $P$, and so $\phi'$ shows $P$ is type \type{$Q$}{relative calibre $\omega$}. To see this take any countably infinite subset $S$ of a $\phi'(q)$. As $\phi'(q)$ is compact, there is an infinite sequence on $S$ with limit in $\phi'(q)$, which is a subset of some $P_n$. By hypothesis, the sequence (and hence $S$) contains an infinite subsequence with upper bound in $P$.

For (c), note that since $P$ is type \type{$\K(M)$}{relative calibre $\omega$}, it is powerfully calibre $\omoneom$ by Lemma~\ref{type_pow_rcalom}. 
Alternatively, $P$ is powerfully calibre $\omoneom$ by applying  Lemma~\ref{prod_CSB}  to $P$ with the finer $2^o$ and CSBS topology.

Finally, if each $P_n$ is $\sigma$-compact, then $M$ is $\sigma$-compact, so $\omega \tq (M,\K(M))$.
Likewise, if each $P_n$ is analytic, then $M$ is analytic, so $\omega^\omega \tq (M,\K(M))$. 
Applying (b) then gives the desired types for $P$.
Note that $\omega$ is calibre $\omega_1$, as is $\omega^\omega$ provided $\omega_1 <\mathfrak{b}$. Hence, under conditions (i) or (ii), $P$ is type \type{calibre $\omega_1$}{relative calibre $\omega$}, and so by Lemmas~\ref{calibre_types} and~\ref{calom1om1omom_prod}, we see that $P$ is indeed calibre $(\omega_1, \omega_1, \omega,\omega)$ and   productively calibre $\omoneom$. 
\end{proof}

\section{Product Examples}\label{sec:exs}

\subsection{Productive and Powerful Examples}
We apply our techniques for showing that directed sets are powerfully or productively calibre $\omoneom$ to two interesting examples from \cite{LV}. In both cases we show the standard topology on the relevant directed sets is second countable, but not CSBS, and so we can not apply the `standard' topological argument. 

\begin{exam}[The ideal of polynomial growth]
Let $\mathcal{I} = \bigcup_{c<\omega} \mathcal{I}_c$ be the ideal of subsets of $\omega$ where $\mathcal{I}_c = \{A\subseteq \omega : |A\cap 2^n| \le n^c,\ \forall n\ge 2\}$. Then

(a) the natural topology on $\mathcal{I}$ is $2^o$ but not CSBS, however

(b) each $\mathcal{I}_c$ is relative calibre $\omega$ in $\mathcal{I}$, hence 

(c) $\mathcal{I}$ is productively and powerfully calibre $\omoneom$.
\end{exam}
\begin{proof}
This is Example~1 from \cite{LV}, where it is shown that each infinite $\mathcal{S}\subseteq \mathcal{I}_c$ has an infinite subset $\mathcal{S}'$ such that $\bigcup\mathcal{S}' \in \mathcal{I}_{c+1}$. 
Hence, each $\mathcal{I}_c$ has relative calibre $\omega$ in $(\mathcal{I},\subseteq)$, which is (b). 
Statement (c) then follows from Corollaries~\ref{sigma-prod} and~\ref{sigma_pow}.

For (a), $\mathcal{I}$ is naturally topologized as a subset of the Cantor set, $\mathbb{P}(\omega)=\{0,1\}^\omega$, and so is second countable. 
We show $\mathcal{I}$ is not CSBS. Consider $A_n = [2^n,2^{n+1})\subseteq\omega$ for each $n<\omega$.  Then each $A_n$ is in $\mathcal{I}$ since it is finite, and the sequence $(A_n)_{n<\omega}$ converges to $\emptyset$ in $\mathcal{I}\subseteq \{0,1\}^\omega$.  Let $(A_{n_k})_{k<\omega}$ be any subsequence.  Fix $c<\omega$ and find a $j<\omega$ such that $2^m > 2m^c$ for all $m\ge j$.  If $A\subseteq \omega$ contains every $A_{n_k}$, then we have:
\[|A\cap 2^{n_j+1}| \ge |A_{n_j}| = 2^{n_j} = \frac{2^{n_j+1}}{2} > \frac{2(n_j+1)^c}{2} = (n_j+1)^c,\]
and so $A$ is not in $\mathcal{I}_c$.  Thus $(A_{n_k})_{k<\omega}$ has no upper bound in $\mathcal{I}$, and $\mathcal{I}$ is not CSBS. 
\end{proof}

The second example is the subject of Section~7 from \cite{LV}. 
Observe that it gives a consistent example of a directed set with calibre $(\omega_1,\omega_1,\omega,\omega)$ but not calibre $(\omega_1,\omega_1,\omega)$.

\begin{exam} Let $\alpha>0$ be a countable ordinal. 
Let $\K^{(\alpha)}$ be the subspace of $\K(2^\omega)$ consisting of (countable) compact subsets with Cantor-Bendixson height $<\alpha$. If $\alpha$ is a successor, then $\K^{(\alpha)} =_T [\ctm]^{<\omega}$.

Now suppose  $\alpha$ is a limit, $\lambda$. Then 

(a) 
 $\K^{(\lambda)}$ is $2^o$ but is not CSBS,  is not calibre $(\omega_1,\omega_1,\omega)$, and for no separable metrizable $M$ do we have $\K(M) \tq \K^{(\lambda)}$; however, 

(b) in some finer topology, $\K^{(\lambda)}$ is $2^o$ and CSB, and

(c)  $\K^{(\lambda)}$ is type 
\type{$\omega^\omega$}{relative calibre $\omega$}, so 

(d) $K^{(\lambda)}$ is powerfully calibre $\omoneom$, and 

(e) if $\omega_1 < \mathfrak{b}$, $K^{(\lambda)}$ is calibre $(\omega_1,\omega_1,\omega,\omega)$ and so productively calibre $\omoneom$.
\end{exam}

\begin{proof} Considered as a subspace of $\K(2^\omega)$, it is known that $\K^{(\alpha)}$ is Borel and hence analytic. 
If $\alpha=\beta+1$ is a successor, then Louveau \& Velickovich, \cite{LV}, point out that there is a perfect set of disjoint compact sets of rank $\beta$, and so $\K^{(\alpha)}$ does indeed have maximal Tukey type.

Now fix limit $\lambda$. Pick an increasing  sequence $(\lambda_n)_n$ converging to $\lambda$. Set $P=\K^{(\lambda)}$ and $P_n=\K^{(\lambda_n)}$. Suppose $(K_m)_m$ is a sequence on $P_n$ with limit $K$ (in $P$), say with rank $\beta <\lambda$. Let $K_\infty= K \cup \bigcup_m K_m$. Then $K_\infty$ is an upper bound in $\K(2^{\omega})$ of the sequence $(K_m)_m$ and has Cantor-Bendixson rank $\le \max (\lambda_m, \beta)+1$, which is strictly less than $\lambda$. So $K_\infty$ is an upper bound of $(K_m)_m$ in $P$. Thus $P$ and the $P_n$'s satisfy the conditions of Lemma~\ref{sigma_rel_csbs}, and claims (b)-(e) follow.

To see $\K^{(\lambda)}$ is not CSBS, choose a sequence $(K_n)_n$ of compact subsets of $2^\omega$ such that $K_n$ has Cantor-Bendixson height $\lambda_n$ and is contained in the $1/n$-ball about $0$. Then $(K_n)_n$ is convergent in $\K^{(\lambda)}$ to an element (namely $\{0\}$) of $\K^{(\lambda)}$, but clearly an upper bound of any subsequence $(K_{n_r})_r$ must have  rank at least that of each $K_{n_r}$, which is $\lambda_{n_r}$, and so the upper bound must have rank at least $\lambda$, and so the upper bound is not in $K^{(\lambda)}$.

Now we show $\K^{(\lambda)}$ is not calibre $(\omega_1,\omega_1,\omega)$. Let $S$ be an uncountable subset of $2^\omega$, considered as a subspace of $\K^{(\lambda)}$. Take any uncountable subset $S_0$ of $S$. We show that some countable subset of $S_0$ is unbounded in $\K^{(\lambda)}$. Well $\cl{S_0}$ is an uncountable compact subset of $2^\omega$, so it contains a Cantor subset $K$. Pick a countable subset $D$ of $S_0$ with $\cl{D} \supseteq K$. Then any upper bound for $D$ would contain $K$, and so would not be in $K^{(\lambda)}$. 

Write $P$ for $\K^{(\lambda)}$, and denote by $\tau$ the standard (Vietoris) topology on $P$.
Now suppose, for a contradiction, that $M$ is separable metrizable  and $\K(M) \tq P$. Fix a Tukey map $\psi : P \to \K(M)$ (that is, $\psi$ maps unbounded subsets of $P$ to unbounded subsets of $\K(M)$). Refine the topology on $P$, by ensuring $f=\psi$ is continuous, to get a new topology $\sigma$ so that  $(P,\sigma)$ is separable metrizable. Take any compact subset $K$ of $(P,\sigma)$. Then $\psi(K)$ is compact in $\K(M)$, and so bounded. As $\psi$ is a Tukey map it follows that $K$ must be bounded. In other words, $(P,\sigma)$ is KSB.

Now, the Cantor set $2^\omega$ is a subset of $P$ and, as a subspace of $(P,\sigma)$, has a separable metrizable topology finer than the usual topology. Since it is uncountable, it cannot be scattered, so it contains a subspace homeomorphic to the rationals. Hence the Cantor set, as a subspace of $(P,\sigma)$, contains a compact subset $K$ which has Cantor-Bendixson height $\ge \lambda$. But any upper bound of $K$ must contain $\bigcup K$ and so have Cantor-Bendixson height $\ge \lambda$, meaning it cannot be in $P=\K^{(\lambda)}$.  This contradicts the fact that $(P,\sigma)$ is KSB and shows no $M$ as above can exist.
\end{proof}

\subsection{Non-Powerful, Non-Productive Example} \label{neither_example}

An easy application of the pressing down lemma establishes:
\begin{lemma} \label{stationary_ce}
Each stationary subset of $\omega_1$ has countable extent.
\end{lemma}


\begin{lemma} \label{k_limit_of_discrete}
If $M$ is separable metrizable, $D$ is a dense subset of $M$, and $K$ is a compact subset of $M$ disjoint from $D$, then there is a discrete subset $D'$ of $D$ whose closure in $M$ is precisely $D'\cup K$.
\end{lemma}

\begin{proof}
Fix a compatible metric $\rho$ for $X$. Let $A$ be a countable dense subset of $K$. Enumerate $A=\{x_n\}_n$ so that each element is repeated infinitely many times. For each $n$, pick $d_n$ in $D$ such that $\rho(d_n,x_n) < 1/n$. Let $D'=\{d_n\}_n$.  Clearly $K$ is contained in the closure of $D'$, but note also that if $(d_{n_k})_k$ converges to $x$ in $M$, then so does $(x_{n_k})_k$, and so all limit points of $D'$ are in $K$.  It follows that the closure of $D'$ is contained in $D' \cup K$, and since $D'$ is disjoint from $K$, then it also follows that $D'$ is discrete.
\end{proof}

In the next example, we write $\CD(Y)$ for the directed set of closed discrete subsets of a space $Y$, ordered by inclusion. 

\begin{exam}\label{ex:neither}
There is a directed set $P$ which is calibre $\omoneom$ but is neither powerfully calibre $\omoneom$ nor productively calibre $\omoneom$. 
\end{exam}

\begin{proof} Let $I$ be the isolated points in $\omega_1$.
Write $\omega_1$ as a disjoint union of stationary (and co-stationary) sets: $\omega_1 = \bigcup_n S_n'$. Let $S_n = S_n' \cup I$, and let $X_n=M(\omega_1,S_n)$, that is, $\omega_1$ with its topology refined by isolating the points in $S_n$. We consider $P = \K(X)$, where $X=\bigoplus_n X_n$.

$\K(S_1)$ is calibre $\omoneom$ since $S_1$ is stationary (and co-stationary), but $\K(X) \times \K(S_1) =_T \K(X \times S_1)$ is not calibre $\omoneom$ according to Lemma~\ref{kx_ce} since $\K(X\times S_1)$ does not have countable extent.  Indeed, $\{(\alpha,\alpha) : \alpha \in S_1\}$ is an uncountable closed discrete subset of $X_1\times S_1$, which is a closed subset of $X\times S_1$, which in turn is a closed subset of $\K(X\times S_1)$.  Similarly $\K(X)^\omega =_T \K(X^\omega)$ does not have calibre $\omoneom$ because $\{ (\alpha, \alpha, \ldots) : \alpha \in \omega_1\}$ is an uncountable closed discrete subset of $\prod_n X_n$, which is a closed subset of $X^\omega$, which in turn is a closed subset of $\K(X^\omega)$.  Therefore, $\K(X)$ is neither productively, nor powerfully calibre $\omoneom$.

It remains to show that $\K(X)$ has calibre $\omoneom$.
Since $\K(X) = \bigcup_n \K(X_1 \oplus \cdots \oplus X_n)$, we see that $\K(X)$ has calibre $\omoneom$ provided each $\K(X_1 \oplus \cdots \oplus X_n)=_T \K(X_1) \times \cdots \times \K(X_n)$ has calibre $\omoneom$. This is immediate from the following three claims:
\begin{itemize}
    \item[(a)] for any stationary (and co-stationary) subset $S$ of $\omega_1$ containing $I$, we have $\K(Y) =_T \CD(S)$, where $Y=M(\omega_1,S)$; 
    
    \item[(b)] for any subsets $S_1, \ldots, S_n$  of $\omega_1$, we have $\CD(S_1) \times \cdots \times \CD(S_n) =_T \CD(S_1 \oplus \cdots \oplus S_n)$; and

    \item[(c)]  if $S_1, \ldots , S_n$ are subsets of $\omega_1$ whose union is co-stationary, then $\CD(S_1 \oplus \cdots \oplus S_n)$ has calibre $\omoneom$.
\end{itemize}
Claim (b) is clear. We prove (a) and deduce (c) with the assistance of an additional claim (d). 

\medskip

{\noindent \bf For (a):}  
Define $\phi_1 : \CD(S) \to \K(Y)$ by $\phi_1(E) = \cl{E}^Y$ and $\phi_2 : \K(Y) \to \CD(S)$ by $\phi_2(K)= K \cap S$.  Then $\phi_2$ is well-defined since, for any compact subset $K$ of $Y$, the topologies $K$ inherits from $Y$ and $\omega_1$ coincide.  To see that $\phi_1$ is well-defined, first notice that for any closed subset $E$ of $S$, the subspaces $\cl{E}^Y$ and $\cl{E}^{\omega_1}$ are equal as sets, but if $E$ is also discrete, then in fact, their topologies coincide as well. By Lemma~\ref{stationary_ce}, each closed discrete subset $E$ of $S$ is countable, so $\cl{E}^{\omega_1}$ is compact and $\phi_1$ is well-defined.

Clearly $\phi_1$ and $\phi_2$ are order-preserving, so it suffices to show that their images are cofinal in $\K(Y)$ and $CD(S)$, respectively.  Since $\phi_2(\phi_1(E))=E$ for every $E$ in $\CD(S)$, we see that, in fact, $\phi_2$ is onto.

Now we check that the image of $\phi_1$ is cofinal.  Let $K$ be a compact subset of $Y$, and let $E_0 = \phi_2(K)\in CD(S)$.  Then $K\cap S = E_0$ is contained in $\phi_1(E_0)$.  We aim to find a closed discrete subset $E_1$ of $S$ such that $\phi_1(E_1)$ contains $K_1 = K \setminus S$. Then $E = E_0 \cup E_1$ will be a closed discrete subset of $S$ such that $\phi_1(E)$ contains $K$.

Note that $K_1$ is compact as a subspace of $Y$ and so also compact as a subpace of $\omega_1$.  Hence, $K_1$ is contained in some interval $[0,\alpha]$.  Now, $[0,\alpha]$ is countable and first countable, so it is separable and metrizable.  Since $S$ contains $I$, then $S\cap [0,\alpha]$ is dense in $[0,\alpha]$, and since $K_1$ is disjoint from $S$, then Lemma~\ref{k_limit_of_discrete} provides a discrete subset $E_1$ of $S\cap [0,\alpha]$ whose closure in $[0,\alpha]$ is $E_1 \cup K_1$.  Since $K_1$ is disjoint from $S$, then $E_1$ is closed in $S$, and $\cl{E_1}^Y = \cl{E_1}^{[0,\alpha]}$.  Hence, $K_1$ is contained in $\phi_1(E_1)$, as desired.

\smallskip
{\noindent \bf For (c):} 
Let $S_1,\ldots, S_n$ be stationary subsets of $\omega_1$ such that $\bigcup_i S_i$ is co-stationary.  Let $\{E_\alpha : \alpha<\omega_1\}$ be an uncountable family of closed discrete subsets of $S_1 \oplus \cdots \oplus S_n$.  Write $E_\alpha = E_\alpha^1 \oplus \cdots \oplus E_\alpha^n$ where $E_\alpha^i$ is a closed discrete subset of $S_i$.  We may assume that for every $i\le n$ and every $\gamma<\omega_1$, there are only countably many $E_\alpha^i$ contained in $[0,\gamma]$ (for example, by adding one point from $S_i \cap [\alpha,\omega_1)$ to $E_\alpha^i$).

By applying claim (d) below repeatedly, we can then find uncountable subsets $A'_1 \supseteq A'_2 \supseteq \cdots \supseteq A'_n$ of $\omega_1$ and $\gamma_1, \ldots, \gamma_n < \omega_1$ such that for each $i$ we have $E_\alpha^i \not\subseteq [0,\gamma_i]$ for all $\alpha$ in $A'_i$ and $\sup E_\alpha^i < \min(E_\beta^i \setminus [0,\gamma_i])$ whenever $\alpha<\beta$ in $A'_i$.  Let $\gamma = \max\{\gamma_1,\ldots,\gamma_n\}$.  We can find an uncountable subset $A_0$ of $A'_n$ such that $E_\alpha^i \not\subseteq [0,\gamma]$ for any $\alpha$ in $A_0$ and any $i$, and such that whenever $\alpha<\beta$ in $A_0$, we have $u_\alpha < \ell_\beta$; here, $u_\alpha = \max_i \sup E_\alpha^i$ and $\ell_\alpha = \min_i \min(E_\alpha^i\setminus [0,\gamma])$ for any $\alpha$ in $A_0$.

Enumerate $S_i \cap [0,\gamma] = \{s_m^i : m<\omega\}$ for each $i\le n$.  We construct a decreasing sequence $(A_m)_{m<\omega}$ of uncountable subsets of $A_0$ and open neighborhoods $U_m^i$ of $s_m^i$ in $S_i$ as follows.  Assume we have already defined $A_{m-1}$ for some $0<m<\omega$.  Since $E_\alpha^i$ is closed discrete in $S_i$ and $s_m^i$ has a countable neighborhood base in $S_i$ for each $i\le n$, we can find an uncountable subset $A_m$ of $A_{m-1}$ and open neighborhoods $U_m^i$ of $s_m^i$ in $S_i$ such that $U_m^i \cap E_\alpha^i \subseteq \{s_m^i\}$ for every $\alpha$ in $A_m$.

For each $m<\omega$, let $C_m = \{u_\alpha : \alpha\in A_m\}$.  Then $C = \bigcap_m \cl{C_m}^{\omega_1}$ is a cub set, and so is the subset $C'$ of all limit points in $C$.  As $\bigcup_i S_i$ is co-stationary, we can find a $u_\infty$ in $C'$ that is not in any $S_i$.  Hence, we can choose an increasing sequence $(\alpha_m)_m$ such that $\alpha_m$ is in $A_m$ and $(u_{\alpha_m})_m$ converges to $u_\infty$.  Since $\sup E_{\alpha_m}^i \le u_{\alpha_m} < \ell_{\alpha_{m+1}} \le \min(E_{\alpha_{m+1}}^i \setminus [0,\gamma])$ for each $m$, and since $u_\infty$ is not in $S_i$, then we see that $\bigcup_m (E_{\alpha_m}^i \setminus [0,\gamma])$ is closed discrete in $S_i$ for each $i$.  On the other hand, $\bigcup_m (E_{\alpha_m}^i \cap [0,\gamma])$ is also closed discrete in $S_i$ because for any $s_k^i$ in $S_i\cap [0,\gamma]$, we have $U_k^i \cap E_{\alpha_m}^i \subseteq \{s_k^i\}$ for every $m\ge k$ (since $\alpha_m \in A_m \subseteq A_k$).  Thus, $\bigcup_m E_{\alpha_m}$ is closed discrete in $S_1\oplus \cdots \oplus S_n$, which completes the proof of claim (c).

\smallskip

\noindent \textbf{Claim (d):}  Let $S$ be a stationary subset of $\omega_1$, and let $A$ be an uncountable subset of $\omega_1$.  If $\{E_\alpha : \alpha\in A\}$ is a family of closed discrete subsets of $S$ such that $B_\gamma = \{\alpha \in A : E_\alpha \subseteq [0,\gamma]\}$ is countable for each $\gamma<\omega_1$, then there is a $\gamma<\omega_1$ and an uncountable $A' \subseteq A \setminus B_\gamma$ such that $\sup E_\alpha < \min(E_\beta \setminus [0,\gamma])$ whenever $\alpha<\beta$ in $A'$.

\noindent \textbf{For (d):} Let $S$, $A$, $E_\alpha$, and $B_\gamma$ be as in the statement of claim (d).  Fix a bijection $f:\omega_1\to A$ such that $\alpha<\beta$ if and only if $f(\alpha)<f(\beta)$.  Each $\alpha$ in $S$ has a neighborhood that intersects $E_{f(\alpha)}$ in at most one point, and since $S$ is stationary, then by the pressing down lemma, we can find a stationary subset $S'$ of $S$ and a $\gamma<\omega_1$ such that $(\gamma,\alpha]\cap E_{f(\alpha)} \subseteq \{\alpha\}$ for each $\alpha$ in $S'$.  For any $\alpha$ in the uncountable set $S'\setminus f^{-1}[B_\gamma]$, we therefore have $\min(E_{f(\alpha)}\setminus [0,\gamma])\ge \alpha$.  Note that, by Lemma~\ref{stationary_ce}, each $E_\alpha$ is countable.  We can then inductively construct an uncountable subset $S''$ of $S'\setminus f^{-1}[B_\gamma]$ such that $\beta > \sup_{\alpha\in S''\cap\beta} \sup E_{f(\alpha)}$ for each $\beta$ in $S''$.  Let $A' = f[S''] \subseteq A \setminus B_\gamma$.  If $f(\alpha) < f(\beta)$ in $A'$, then we have $\sup E_{f(\alpha)} < \beta \le \min(E_{f(\beta)}\setminus [0,\gamma])$, which completes the proof of  (d).
\end{proof}

\subsection{Powerful not Productive, and $\ssum$-Products} \label{counter_example}

We now show that there is a directed set which is powerfully calibre $\omoneom$ but not productively calibre $\omoneom$. We simultaneously show that the restriction in Theorem~\ref{Sigma_cosmic_CSB} to cosmic directed sets with CSBS is not arbitrary.
Indeed, if $\sum P_\alpha$ has calibre $(\omega_1,\omega)$, then every projection has the same calibre; in particular, every countable subproduct of the $P_\alpha$'s has calibre $(\omega_1,\omega)$, so it is natural to hope that this necessary condition for a $\ssum$-product to have calibre $(\omega_1,\omega)$ is also sufficient. This is not the case, even when all the directed sets, $P_\alpha$, are equal.

\begin{exam}\label{Sigma_ctble}
There is a directed set $P$ such that:

(i) the $\ssum$-product, $P \times \sum \omega^{\omega_1}$ does not have calibre $(\omega_1,\omega)$, but every countable subproduct does have calibre $(\omega_1,\omega)$; and 

(ii) $\sum P^{\omega_1}$ does not have calibre $(\omega_1,\omega)$, but $P^\omega$ has calibre $(\omega_1,\omega)$.

\smallskip

Hence both $P$ and $\sum \omega^{\omega_1}$ are powerfully calibre $\omoneom$ but not productively calibre $\omoneom$.
\end{exam}

\begin{proof}
Assume first that there is a directed set $P$ satisfying (i) and (ii). From (ii) $P$ is powerfully calibre $\omoneom$. Since $\sum \omega^{\omega_1}$ is clearly Tukey equivalent to its countable power, from Theorem~\ref{Sigma_cosmic_CSB} we see that $\sum \omega^{\omega_1}$ is also powerfully calibre $\omoneom$. Now (i) says that neither P nor $\sum \omega^{\omega_1}$ are  productively calibre $\omoneom$, and the final claim is established.

So we need to show that there is a directed set $P$ with properties (i) and (ii). Let $X$ be any subset of the reals which is totally imperfect and has size $\omega_1$ (say an $\omega_1$-sized  subset of a Bernstein set). Index $X=\{x_\alpha : \alpha < \omega_1\}$. For each $\alpha$ let $U_\alpha = \{ x_\beta : \beta \ge \alpha\}$, and refine the standard topology on $X$ (inherited from $\R$) by adding the sets $U_\alpha$, for $\alpha$ in $\omega_1$, as open sets. We write $X$ for the set $X$ with this topology. It is a Hausdorff space.

Let $P=\K(X)$. Then $P$ is a Hausdorff topological directed set with CSB and DK. We verify two claims:
\begin{itemize}
    \item[(a)] $\K(X) \times \sum \omega^{\omega_1}$ does not have calibre $(\omega_1,\omega)$, but
    \item[(b)] $\K(X)^\omega$ is calibre $(\omega_1,\omega)$.
\end{itemize}
Since $\K(X) \tq \omega$, we see that $\sum \K(X)^{\omega_1} \tq \K(X) \times \sum \omega^{\omega_1}$ and $\K(X)^\omega \tq \K(X) \times \omega^\omega$. Hence, claims (i) and (ii) follow from (a) and (b). 

\medskip

\noindent {\bf For (a):} By Lemma~\ref{sigma_DK} below,  and Lemma~\ref{om1om_to_ce}, we know it suffices to show $P \times \sum \omega^{\omega_1}$ contains an uncountable closed discrete set.

For each $\alpha$, fix a decreasing local base, $(B_{\alpha,n})_n$, at $x_\alpha$ such that $B_{\alpha,1} \subseteq U_\alpha$.  Define $\vy_\alpha = (y_{\alpha,\beta})_\beta$ in $\sum \omega^{\omega_1}$ by 
$y_{\alpha,\beta} = \min\{n : x_\alpha\not\in B_{\beta,n}\}$ if $\beta<\alpha$ and $0$ for $\beta \ge \alpha$.
Then let $E=\{ (x_\alpha, \vy_\alpha) : \alpha < \omega_1 \} \subseteq P \times \sum \omega^{\omega_1}$.  Obviously $E$ is uncountable.
We show $E$ is closed discrete in $P\times\sum \omega^{\omega_1}$

Since $X$ is closed in $P = \K(X)$, it suffices to show $E$ is closed and discrete in $X \times \sum \omega^{\omega_1}$.  Fix $(x_\beta, \vz)$ in $X \times \sum \omega^{\omega_1}$.  We need to find an open neighborhood of $(x_\beta,\vz)$ that contains at most one member of $E$. Let $z_\beta = \pi_\beta(\vz)$ and consider $V = V_{\beta,\vz} = B_{\beta,z_\beta} \times \pi_\beta^{-1} \{z_\beta\}$, which contains $(x_\beta,\vz)$.

Suppose $(x_\alpha,\vy_\alpha)$ is in $E \cap V$.  Then $x_\alpha$ is in $B_{\beta,z_\beta} \subseteq U_\beta$, so $\beta \le \alpha$.  If $\beta < \alpha$, then $y_{\alpha,\beta} = \min \{n\in\N : x_\alpha \not\in B_{\beta,n}\} > z_\beta$, which contradicts that $\vy_\alpha$ is in $\pi_\beta^{-1} \{z_\beta\}$.  Thus, $\alpha = \beta$, so $V$ contains at most one member of $E$, namely $(x_\beta, \vy_\beta)$, which proves the claim (since $X$ is $T_1$).

\smallskip

\noindent {\bf For (b):} We will show that $P=\K(X)^\omega$ is first countable and hereditarily ccc (all discrete subsets are countable). Then $P$ has countable extent, is first countable and CSB, and so has calibre $(\omega_1,\omega)$ by Lemma~\ref{ce_to_om1om}. 

That $\K(X)$ (and hence $\K(X)^\omega$) is first countable follows from Lemma~\ref{kx_1ctble}, so we focus on showing $\K(X)^\omega$ contains no uncountable discrete subsets.
It is a standard fact (and straightforward to check) that the countable power $Y^\omega$ of a space $Y$ is hereditarily ccc if and only if every finite power if $Y$ is hereditarily ccc. So we show for every $n$ in $\N$ the space $\K(X)^n$ is hereditarily ccc.

Recall that basic neighborhoods in $\K(X)$ (with the Vietoris topology) have the form $\langle U_1,\ldots, U_n \rangle = \{K\in\K(X) : K \subseteq \bigcup_{i=1}^n U_i,\ K\cap U_i \neq \emptyset \ \forall i=1,\ldots,n\}$ for some $n$ in $\N$ and $U_i$ open in $X$; further, the open sets $U_1, \ldots, U_n$ can be assumed to come from any specified base for $X$.

Suppose, for a contradiction,  $\K(X)^n$ is not hereditarily ccc.  Then there is an uncountable discrete subset $\mathcal{S} = \{\mathbf{K}_\lambda = (K^\lambda_1,\ldots,K^\lambda_n) : \lambda < \omega_1\}$ of $\K(X)^n$, enumerated in a one-to-one fashion.  So for each $\lambda < \omega_1$, there is a basic open $\mathbf{V}_\lambda = \prod_{i=1}^n \langle V^\lambda_{i,1}, \ldots, V^\lambda_{i,m_\lambda} \rangle$ such that $\mathbf{V}_\lambda \cap \mathcal{S} = \{\mathbf{K}_\lambda\}$.

Let $\B$ be a countable base for the usual topology on $X = \{x_\alpha : \alpha < \omega_1 \} \subseteq \R$, and define $C(\alpha) = \{x_\beta : \beta < \alpha\} = X \setminus U_\alpha$ for each $\alpha < \omega_1$.  Then we can choose each $V^\lambda_{i,j}$ to have the form $V^\lambda_{i,j} = B^\lambda_{i,j} \setminus C(\alpha^\lambda_{i,j})$ for some $B^\lambda_{i,j }\in \B$ and $\alpha^\lambda_{i,j} < \omega_1$. By replacing $\mathcal{S}$ with an appropriate uncountable subset, we may assume there is some $m\in\N$ and $B_{i,j}\in\B$ for $i\in\{1,\ldots,n\}$ and $j\in\{1,\ldots,m\}$ such that $m_\lambda = m$ and $B^\lambda_{i,j} = B_{i,j}$ for each $\lambda < \omega_1$.  Then $\mathbf{V}_\lambda$ is uniquely determined 
by the corresponding $n\times m$ matrix $(\alpha^\lambda_{i,j})_{i,j}$.

Consider the countable subset $\{\mathbf{K}_\ell : \ell<\omega\}$ of $\mathcal{S}$.  If there are $k<\ell<\omega$ such that $\alpha^k_{i,j} \le \alpha^\ell_{i,j}$ for each $i,j$, then $V^\ell_{i,j} = B_{i,j} \setminus C(\alpha^\ell_{i,j}) \subseteq B_{i,j} \setminus C(\alpha^k_{i,j}) = V^k_{i,j}$, which implies that $\mathbf{V}_\ell \subseteq \mathbf{V}_k$.  But that is a contradiction since $\mathbf{K}_\ell$ is in $\mathbf{V}_\ell \setminus \mathbf{V}_k$.  Hence, for any $k<\ell<\omega$, we have $\alpha^k_{i,j} > \alpha^\ell_{i,j}$ for some $i = i(k,\ell) \in\{1,\ldots,n\}$ and $j = j(k,\ell) \in \{1,\ldots,m\}$. 
By Ramsey's Theorem, there is an infinite subset $A$ of $\omega$, an $i \in \{1,\ldots,n\}$, and a $j \in \{1,\ldots,m\}$ such that for any $k<\ell$ in $A$, we have $i(k,\ell) = i$ and $j(k,\ell)=j$.  But then $\{\alpha^\ell_{i,j} : \ell\in A\}$ forms an infinite decreasing sequence in $\omega_1$, which contradicts that $\omega_1$ is well-ordered.  Therefore, such $\mathcal{S}$ cannot exist.
\end{proof}

\begin{lemma} \label{sigma_DK}
$\sum_\alpha P_\alpha$ is DK if every $P_\alpha$ is DK.
\end{lemma}

\begin{proof}
For any $\mathbf{p} = (p_\alpha)_\alpha$ in $\sum P_\alpha$, the down set, $\down{\mathbf{p}}$, of $\mathbf{p}$ is contained in $\sum_\alpha P_\alpha$ and is equal to $\prod_\alpha (\down{p_\alpha})$ (in $\prod_\alpha P_\alpha$).  Hence, if each $\down{p_\alpha}$ is compact, then so is $\down{\mathbf{p}}$.
\end{proof}

\section{Isbell's Examples, and $\sum \omega^{\omega_1}$}\label{sec:isbell}

Recall Isbell's ten directed sets consisted of a first group of $\mathbf{1}$, $\omega$, $\omega_1$, $\omega \times \omega_1$, $[\omega_1]^{<\omega}$, and a second group, $\omega^\omega$, $\sum \omega^{\omega_1}$, $\text{NWD}$, $\mathcal{Z}_0$ and $\ell_1$.   Isbell observed that, among the second group, $\omega^\omega$ and $\sum \omega^{\omega_1}$ have a certain property preserved by Tukey quotients, but showed $\mathcal{Z}_0$ and $\ell_1$ do  not have this property. 
Fremlin added $\mathcal{E}_\mu$, the ideal of all compact subsets of $I$ which are null sets (with respect to the standard Lebesgue measure, $\mu$). Among other things he showed $\mathcal{E}_\mu$ is strictly above $\omega^\omega$, below both $\text{NWD}$ and $\mathcal{Z}_0$, and these latter two directed sets are below $\ell_1$. \cite{TodSol} and \cite{Matrai} showed that the mentioned relations between these five directed sets ($\omega^\omega$, $\mathcal{E}_\mu$, $\text{NWD}$, $\mathcal{Z}_0$ and $\ell_1$) are all strict, and no other relations hold. 

 We show $\sum \omega^{\omega_1}$ is strictly above $\omega^\omega$ and is incomparable with each of $\mathcal{E}_\mu$, $\text{NWD}$, $\mathcal{Z}_0$ and $\ell_1$. 
Clearly $\sum\omega^{\omega_1} \tq \omega^\omega$. Recalling that each of $\omega^\omega$, $\mathcal{E}_\mu$, $\text{NWD}$, $\mathcal{Z}_0$ and $\ell_1$ is $2^o$ and CSBS, and $\mathcal{E}_\mu$ is below all of $\text{NWD}$, $\mathcal{Z}_0$ and $\ell_1$, our claim follows from the next two results.

\begin{proposition}\label{no_csbs_above}
Let $P$ be a $2^o$ and CSBS topological directed set. Then $P \not\tq \sum \omega^{\omega_1}$.
\end{proposition}
\begin{proof} Suppose, for a contradiction, that $P \tq \sum\omega^{\omega_1}$.   
Let $Q=\K(X)$ be as in Example~\ref{Sigma_ctble}. Then $Q$ is first countable, hereditarily ccc and CSBS, but $Q \times \sum \omega^{\omega_1}$ is not calibre $\omoneom$. Now, on the one hand, $Q \times P \tq Q \times \sum \omega^{\omega_1}$, and so $Q \times P$ is not calibre $\omoneom$.  On the other hand, $Q \times P$ is first countable, hereditarily ccc (the product of a first countable and hereditarily ccc space with a $2^o$ space is first countable and hereditarily ccc) and CSBS, and so must be calibre $\omoneom$, which is a contradiction.
\end{proof}

Define $\mathop{cov}(\mathcal{N})$, the covering number of the ideal of null sets, to be the smallest cardinal $\kappa$ such that there is a cover of $I=[0,1]$ by $\kappa$-many null sets. Then $\omega_1 \le \mathop{cov}(\mathcal{N}) \le \ctm$, and it is consistent, for example if $\text{MA}+\neg\text{CH}$ holds, that $\omega_1 < \mathop{cov}(\mathcal{N})$.
\begin{proposition}[Assuming $\omega_1 < \mathop{cov}(\mathcal{N})$] \label{Emu}\ 
$\sum \omega^{\omega_1} \not\tq \mathcal{E}_\mu$.
\end{proposition}
\begin{proof}
Suppose, for a contradiction, $\phi : \sum \omega^{\omega_1} \to \mathcal{E}_\mu$ is order preserving and cofinal. For each countably infinite subset $S$ of $\omega_1$, define $e_S : \omega^S \to  \sum \omega^{\omega_1}$ by $e((x_\alpha)_{\alpha \in S}) =(x_\alpha)_{\alpha \in \omega_1}$ where $x_\alpha=0$ if $\alpha \notin S$. Then let $\phi_S = \phi \circ e_S : \omega^S \to \mathcal{E}_\mu$ and define $A_S = \bigcup \phi_S (\omega^S)$. Note that $\phi_S$ is order-preserving, so $A_S$ is analytic because it is separable metrizable and has an $\omega^\omega$-ordered compact cover. 

For infinite $\alpha$ in $\omega_1$, let $\phi_\alpha=\phi_{[0,\alpha)}$ and $A_\alpha=A_{[0,\alpha)}$. Then $\{A_\alpha : \omega \le \alpha < \omega_1\}$ is a cover of $I$ by measurable (because analytic) sets. By hypothesis there is an $\alpha_0$ such that $A=A_{\alpha_0}$ has strictly positive measure, say $\delta>0$. Let $\phi_0=\phi_{\alpha_0}$. Enumerate $[0,\alpha_0)=\{\alpha_1, \alpha_2, \ldots\}$. 
Define $\sigma = (\sigma_{\alpha_n})_n$ in $\omega^{[0,\alpha_0)}$ and subsets $C_n$ of $A$ as follows.
First, note that $A$ is the increasing union, over $r$, of the sets $\bigcup \{\phi_0 (\tau) : \tau \in \omega^{[0,\alpha_0)} \text{ and } \tau_{\alpha_1} \le r\}$. 
So we can pick $\sigma_{\alpha_1}$ sufficiently large that $C_1=\bigcup \{\phi_0 (\tau) : \tau \in \omega^{[0,\alpha_0)} \text{ and } \tau_{\alpha_1} \le \sigma_{\alpha_1}\}$ has outer measure  $>\delta/2$. Then recursively, pick $\sigma_{\alpha_n}$ so that $C_n=\bigcup \{\phi_0 (\tau) : \tau \in \omega^{[0,\alpha_0)} \text{ and } \tau_{\alpha_i} \le \sigma_{\alpha_i}, \, \forall i \le n\}$  has outer measure  $>\delta/2$. 

Let $D=\{x_n :  n \in \N\}$ be a countable dense subset of $C=\bigcap_n C_n$. For each $n$ select $\sigma_n$ in $\omega^{[0,\alpha_0)}$ such that $x_n$ is in $\phi_0(\sigma_n)$.  Since $x_n$ is in $C_n$ and $\phi_0$ is order-preserving, we may select $\sigma_n$ so that it coincides with $\sigma$ on $\{\alpha_1, \ldots, \alpha_n\}$. Then the  $\sigma_n$'s converge to $\sigma$, so there is an upper bound $\sigma_\infty$ of the $\sigma_n$'s in $\omega^{[0,\alpha_0)}$.

But, as $\phi_0$ is order-preserving, $D \subseteq \bigcup_n \phi_0(\sigma_n) \subseteq \phi_0(\sigma_\infty)$, and so the smallest compact subset of $I$ containing $\phi_0(\sigma_\infty)$ must contain $C$ which has strictly positive measure ($\ge \delta/2$). This contradicts the fact that $\phi_0$ maps into $\mathcal{E}_\mu$, the compact measure zero subsets of $I$.
\end{proof}


\begin{thebibliography}{99}

\bibitem{Eng} R. Engelking, \emph{General topology}, Revised and completed edition. Sigma Series in Pure Mathematics, vol. 6. Heldermann Verlag, Berlin, 1989.

\bibitem{FG} Z. Feng, P. Gartside, \emph{Directed Sets of Topology -- Tukey Representation and Rejection}, to appear in Rev. R. Acad. Cienc. Exactas Fís. Nat. Ser. A Mat. (RACSAM)

\bibitem{Frem_An} D. Fremlin, \emph{The partially ordered sets of measure theory and Tukey's ordering}, Note Mat. 11 (1991),  177-214. Dedicated to the memory of
Professor Gottfried K\"{o}the.

\bibitem{PG_AM} P. Gartside, A. Mamatelashvili,
\emph{Tukey order on compact subsets of separable metric spaces}, Journal of Symbolic Logic 81 no. 1, (2016), 181-200. 

\bibitem{KCal} P. Gartside, A.  Mamatelashvili, \emph{Tukey order, calibres and the rationals}, 
Ann. Pure Appl. Logic 172 (2021), no.1, Paper No. 102873, 18 pp.


\bibitem{Isbell} J. Isbell,   \emph{Seven Cofinal Types}, Journal of the London Mathematical Society (1972), Volumes 2-4, Issue 4,  651-654. 



\bibitem{LV} A. Louveau, B. Velickovic, \emph{Analytic Ideals and Cofinal Types}, Annals of Pure and Applied Logic, 99 (1999), 171-195.

\bibitem{Matrai} T. M\'{a}trai, 
\emph{On a $\sigma$-ideal of compact sets}, 
Topology Appl. 157 (2010), no.8, 1479-1484.



\bibitem{TodSol}  S. Solecki,  S. Todorcevic, \emph{Cofinal types of topological directed orders}, 
Annales de l'Institut Fourier, Volume 54 (2004) no. 6,  1877-1911. 


\bibitem{Tod} S. Todorcevic, \emph{Directed Sets and Cofinal Types}, Trans. Amer. Math. Soc., 290 (1985), 711-723.



\end{thebibliography}
\end{document}